\documentclass[twoside,10pt]{amsart}

\usepackage{amsthm,amsmath,amsfonts,epsfig,amssymb}
\usepackage[T1]{fontenc}
\usepackage{graphicx}
\usepackage{enumerate}
\usepackage{xy}
\usepackage{graphics}
\usepackage[colorlinks=true]{hyperref} 

\xyoption{all}
\entrymodifiers={+!!<0pt,\fontdimen22\textfont2>}

\newtheorem{thm}{Theorem}[section]

\newtheorem{prop} [thm]{Proposition}
\newtheorem{lem} [thm]{Lemma}
\newtheorem{coro}[thm]{Corollary}

\newtheorem*{thm*}{Theorem}
\theoremstyle{definition}
\newtheorem{defin}[thm]{Definition}

\theoremstyle{remark}
\newtheorem{rem}[thm]{Remark}
\newtheorem{exem}[thm]{Example}

\def\Z{\mathbb{Z}}
\def\N{\mathbb{N}}
\def\A{\mathbb{A}}
\def\OO{\mathcal{O}}
\def\L{\mathcal{L}}

\def\id{\mathrm{id}} 

\def\gm{\mathbb{G}_m} 
\def\gmk#1{\mathbb{G}_{m,#1}} 

\def\uHom{\underline{\mathrm{Hom}}} 

\def\spec#1{\mathrm{Spec}(#1)}

\def\KM{\mathrm{K}^{\mathrm{M}}} 
\def\KMW{\mathrm{K}^{\mathrm{M\hspace{-.2ex}W}}} 

\def\sKM{\mathbf{K}^{\mathrm{M}}} 
\def\sKMW{\mathbf{K}^{\mathrm{M\hspace{-.2ex}W}}} 

\newcommand{\tchi}[2]{%
  \mathchoice{\widetilde{\mathrm{CH}}_{#2}^{\raisebox{-.5ex}{$\scriptstyle#1$}}}
             {\widetilde{\mathrm{CH}}_{#2}^{\raisebox{-.7ex}{$\scriptstyle#1$}}}
             {}
             {}
}

\def\chst#1#2#3#4{\tchi{#1}{#2}(#3,#4)}

\def\GW{\mathrm{GW}} 

\def\H{\mathrm{H}}
\def\HMW{\mathrm{H}_{\mathrm{MW}}}

\def\tZ{\tilde\Z}

\def\tZX#1{{\tilde {\mathrm{c}}}(#1)}

\def\ome#1#2{\omega_{#1/#2}}

\def\sm#1{\mathrm{Sm}_{#1}}

\def\cor#1{\widetilde{\mathrm{Cor}}_{#1}}

\def\Adm{\mathcal{A}}

\def\Tr{\mathrm{Tr}}

\def\Zar{\mathrm{Zar}}

\def\psh#1{\widetilde{\mathrm{PSh}}_{#1}}

\def\Cstar{\mathrm{C}_*^{\mathrm{sing}}}

\begin{document}

\title{A comparison theorem for MW-motivic cohomology}

\author{Baptiste Calm\`es}\email{baptiste.calmes@univ-artois.fr}\address{Laboratoire de math\'ematiques de Lens \\
Facult\'e des sciences Jean Perrin \\
Universit\'e d'Artois\\
Rue Jean Souvraz SP 18\\
62307 Lens Cedex\\
France}
\author{Jean Fasel}\email{jean.fasel@gmail.com}\address{Institut Fourier-UMR 5582 \\ Universit\'e Grenoble-Alpes   \\CS 40700 \\
38058 Grenoble Cedex 9 \\
France}

\thanks{The first author acknowledges the support of the French Agence Nationale de la Recherche (ANR) under reference ANR-12-BL01-0005}

\begin{abstract}
Let $k$ be an infinite perfect field. We prove that $\HMW^{n,n}(\spec L,\Z)=\KMW_n(L)$ for any finitely generated field extension $L/k$ and any $n\in\Z$.
\end{abstract}



\maketitle

\pagenumbering{arabic}


\setcounter{tocdepth}{1}
\tableofcontents

\section*{Introduction}

This paper is the fourth of a series of papers (\cite{Calmes14b}, \cite{Deglise16} and \cite{Fasel16b}) devoted to the study of MW-motivic cohomology, which is a generalization of ordinary motivic cohomology. Our main purpose here is to compute the MW-motivic cohomology group of a field in bidegree $(n,n)$, namely the group $\HMW^{n,n}(L,\Z)$. In \cite[Theorem 4.2.3]{Deglise16}, we defined a graded ring homomorphism
\[
\Phi:\sKMW_*\to \bigoplus_{n\in\Z} \mathbf{H}_{\mathrm{MW}}^{n,n}
\]
where the left-hand side is the unramified Milnor-Witt $K$-theory sheaf constructed in \cite[\S 3]{Morel08} and the right-hand side is the Nisnevich sheaf associated to the presheaf $U\mapsto \HMW^{n,n}(U,\Z)$. The homomorphism $\Phi$ is obtained via a morphism of sheaves $\gm^{\wedge n}\to \mathbf{H}_{\mathrm{MW}}^{n,n}$ and the right-hand side has the property to be strictly $\A^1$-invariant \cite[Proposition 1.2.11, Theorem 3.2.9]{Deglise16}. It follows that $\Phi$ is then the universal morphism described in \cite[Theorem 3.37]{Morel08}. In this article, we prove that $\Phi$ is an isomorphism. This can be checked on finitely generated field extensions of the base field $k$ (\cite[Theorem 1.12]{Morel08}) and thus our main theorem takes the following form.

\begin{thm*}
Let $L/k$ be a finitely generated field extension with $\mathrm{char}(k)\neq 2$. Then, the homomorphism of graded rings
\[
\Phi_L:\bigoplus_{n\in\Z} \KMW_n(L)\to \bigoplus_{n\in\Z} \HMW^{n,n}(L,\Z).
\]
is an isomorphism.
\end{thm*}

The isomorphism in the theorem generalizes the result on (ordinary) motivic cohomology in the sense that the diagram commutes
\[
\xymatrix{\bigoplus_{n\in\Z} \KMW_n(L)\ar[r]^-{\Phi_L}\ar[d] & \bigoplus_{n\in\Z} \HMW^{n,n}(L,\Z)\ar[d] \\
\bigoplus_{n\in\N} \KM_n(L)\ar[r] & \bigoplus_{n\in\N} \H^{n,n}(L,\Z)}
\]
where the vertical homomorphisms are the ``forgetful'' homomorphisms and the bottom map is the isomorphism produced by Nesterenko-Suslin-Totaro. Unsurprisingly, our proof is very similar to theirs but there are some essential differences. For instance, the complex in weight one, denoted by $\tZ(1)$, admits an epimorphism to $\sKMW_1$ paralleling the epimorphism $\Z(1)\to \sKM_1$. However, we are not able to prove directly  that the kernel of the epimorphism $\tZ(1)\to \sKMW_1$ is acyclic. We are thus forced to compute by hand its cohomology at the right spot in Proposition \ref{prop:inductionstep}. This result being obtained, we then prove that $\Phi$ respects transfers for finitely generated field extensions. This is obtained in Theorem \ref{thm:respect} using arguments essentially identical to \cite[Lemma 5.11]{Mazza06} or \cite[Lemma 9.5]{Neshitov14}.

The paper is organized as follows. In Section \ref{sec:MWmotivic}, we review the basics of MW-motivic cohomology needed in the paper, adding useful results. For instance, we prove a projection formula in Theorem \ref{thm:projection} which is interesting on its own. In Section \ref{sec:main}, we proceed with the proof of our main theorem, starting with the construction to a left inverse of $\Phi$. We then pass to the proof that $\Phi$ is an isomorphism in degree $1$, which is maybe the most technical result of this work. As already mentioned above, we then conclude with the proof that $\Phi$ respects transfers, obtaining as a corollary our main result.  


\subsection*{Conventions}
The schemes are separated of finite type over some perfect field $k$ with $\mathrm{char}(k)\neq 2$. If $X$ is a smooth connected scheme over $k$, we denote by $\Omega_{X/k}$ the sheaf of differentials of $X$ over $\spec k$ and write $\ome{X}{k}:=\det\Omega_{X/k}$ for its canonical sheaf. In general we define $\ome{X}{k}$ connected component by connected component. We use the same notation if $X$ is the localization of a smooth scheme at any point. If $k$ is clear from the context, we omit it from the notation. If $f:X\to Y$ is a morphism of (localizations of) smooth schemes, we set $\omega_{f}=\ome{X}{k}\otimes f^*\ome{Y}{k}^\vee$. If $X$ is a scheme and $n\in\N$, we denote by $X^{(n)}$ the set of codimension $n$ points in $X$.


\section{MW-motivic cohomology}\label{sec:MWmotivic}

The general framework of this article is the category of finite MW-correspondences as defined in \cite[\S 4]{Calmes14b}. We briefly recall the construction of this category for the reader's convenience. If $X$ and $Y$ are smooth connected schemes over $k$, we say that a closed subset $T\subset X\times Y$ is admissible if its irreducible components (endowed with their reduced structure) are finite and surjective over $X$. The set $\Adm (X,Y)$ of admissible subsets of $X\times Y$ can be ordered by inclusions, and we can consider it as a category. For any $T\in \Adm(X,Y)$, we can consider the Chow-Witt group
\[
\chst {d_Y}T{X\times Y}{\omega_Y}
\]
where $d_Y$ is the dimension of $Y$ and $\omega_Y=p_Y^*\omega_{Y/k}$ with $p_Y:X\times Y\to Y$ the projection. If $T\subset T^\prime$ for two admissible subsets, then we consider the extension of support homomorphism
\[
\chst {d_Y}T{X\times Y}{\omega_Y}\to \chst {d_Y}{T^\prime}{X\times Y}{\omega_Y}
\]
and set $\cor k(X,Y)=\lim_{T\in \Adm(X,Y)}\chst {d_Y}T{X\times Y}{\omega_Y}$. The composition of finite MW-correspondences is obtained using the product of cycles in Chow-Witt groups with supports (\cite[\S 4.2]{Calmes14b}) and we obtain the category $\cor k$ whose objects are smooth schemes and morphisms are $\cor k(X,Y)$. The exterior product endows $\cor k$ with the structure of a symmetric monoidal category.

Having this category at hand, we may define the category of MW-presheaves $\psh k$ as the category of additive functors $\cor k\to \mathrm{Ab}$. For any smooth scheme $X$, we can define the presheaf $\tZX X$ by $Y\mapsto \cor k(Y,X)$ for any $Y$ and thus obtain the Yoneda embedding functor $\tZX {\_}:\cor k\to \psh k$. The category $\psh k$ is a symmetric monoidal category, with tensor product $\otimes$ uniquely defined by the property that the Yoneda embedding is monoidal, i.e. we have $\tZX X\otimes \tZX Y=\tZX {X\times Y}$. One can also define an internal Hom functor $\uHom$ which is characterized by the property that $\uHom (\tZX X,F)=F(X\times \_)$ for any $F\in \psh k$. 

Recall next that we have a functor $\tilde\gamma:\sm k\to \cor k$ which is the identity on objects and associates to a morphism of schemes the finite MW-correspondence described in \cite[\S 4.3]{Calmes14b} (which is basically the graph). This yields a functor $\tilde\gamma_*:\psh k\to \mathrm{PSh}_k$ where the latter is the category of presheaves (of abelian groups) on $\sm k$. As usual, we say that a presheaf with MW-transfer $F$ is a sheaf in a topology $\tau$, and we write that $F$ is a $\tau$-sheaf with MW-transfers, if $\tilde\gamma_*(F)$ is a sheaf in this topology. Usually, we consider either the Zariski or the Nisnevich topology on $\sm k$. Interestingly, the representable presheaves $\tZX X$ are Zariski sheaves with MW-transfers (\cite[Proposition 5.11]{Calmes14b}) but not Nisnevich sheaves with transfers (\cite[Example 5.12]{Calmes14b}). However, one can show that the sheaf associated to $F\in \psh k$ can be endowed with a (unique) structure of a sheaf with MW-transfers (\cite[Proposition 1.2.11]{Deglise16}). Note that it is easy to check that if $F$ is a $\tau$-sheaf with MW-transfers, then $\uHom (\tZX X,F)$ is also a $\tau$-sheaf with MW-transfers.

\subsection{Motivic cohomology}

Let $\tZ \{1\}$ be the Zariski sheaf with MW-transfers which is the cokernel of the morphism 
\[
\tZX k\to \tZX {\gmk k}
\]
induced by the unit in $\gmk k$. For any $q\in \Z$, we consider next the Zariski sheaf with MW-transfer $\tZ\{q\}$ defined by 
\[
\tZ\{q\}=\begin{cases} \tZ \{1\}^{\otimes q} & \text{if $q\geq 0$.} \\
\uHom(\tZ \{1\}^{\otimes q}, \tZX k) & \text{if $q<0$.} \end{cases}
\]

Let now $\Delta^{\bullet}$ be the cosimplicial object whose terms in degree $n$ are
\[
\Delta^n=\spec {k[t_0,\ldots,t_n]/(\sum t_i-1)}
\]
and with usual face and degeneracy maps. For any presheaf $F\in \psh k$, we obtain a simplicial presheaf $\uHom(\tZX {\Delta^{\bullet}},F)$ whose associated complex of presheaves with MW-transfers is denoted by $\Cstar(F)$. If $F$ is further a $\tau$-sheaf with MW-transfers, then $\Cstar F$ is a complex of sheaves with MW-transfers. In particular, $\tZ(q):=\Cstar \tZ\{q\}$ is such a complex and we have the following definition.

\begin{defin}
For any $p,q\in \Z$ and any smooth scheme $X$, we set
\[
\HMW^{p,q}(X,\Z)=\mathbb{H}^p_{\Zar}(X,\tZ (q)).
\]
\end{defin}

\begin{rem}
In \cite[\S 3.2.13, Definition 3.3.5]{Deglise16}, the motivic cohomology groups are defined using the complexes associated to the simplicial Nisnevich sheaves with MW-transfers constructed from the Nisnevich sheaves with transfers associated to the presheaves $\tZ (q)$. The two definitions coincide by \cite[Corollary 4.0.5]{Fasel16b}.
\end{rem}

The complexes $\Cstar \tZ(q)$ are in fact complexes of Zariski sheaves of $\KMW_0(k)$-modules (\cite[\S 5.3]{Calmes14b}), and it follows that the MW-motivic cohomology groups are indeed $\KMW_0(k)$-modules. These modules are by construction contravariantly functorial in $X$. Moreover, for any $p,q\in \Z$, we have a homomorphism of $\KMW_0(k)$-modules
\[
\HMW^{p,q}(X,\Z)\to \H^{p,q}(X,\Z)
\]
where the latter denotes the ordinary motivic cohomology group of $X$, with $\H^{p,q}(X,\Z)=0$ for $q<0$ and the $\KMW_0(k)$-module structure on the right-hand side is obtained via the rank homomorphism $\KMW_0(k)\to \Z$ (\cite[\S 6.1]{Calmes14b}).

Even though MW-motivic cohomology is defined a priori only for smooth schemes, it is possible to extend the definition to limits of smooth schemes, following the usual procedure (described for instance in \cite[\S 5.1]{Calmes14b}). In particular, we can consider MW-motivic cohomology groups $\HMW^{p,q}(L,\Z)$ for any finitely generated field extension $L/k$. We will use this routinely in the sequel without further comments.

\subsection{The ring structure}

The definition of MW-motivic cohomology given in \cite[Definition 3.3.5]{Deglise16} immediately yields a (bigraded) ring structure on MW-motivic cohomology
\[
\HMW^{p,q}(X,\Z)\otimes \HMW^{p^\prime,q^\prime}(X,\Z)\to \HMW^{p+p^\prime,q+q^\prime}(X,\Z)
\]
fulfilling the following properties.
\begin{enumerate}
\item The product is (bi-)graded commutative in the sense that 
\[
\HMW^{p,q}(X,\Z)\otimes \HMW^{p^\prime,q^\prime}(X,\Z)\to \HMW^{p+p^\prime,q+q^\prime}(X,\Z)
\] 
is $(-1)^{pp^\prime}\langle (-1)^{qq^\prime}\rangle$-commutative. In particular, $\HMW^{0,0}(X,\Z)$ is central and the $\KMW_0(k)$-module structure is obtained via the ring homomorphism 
\[
\KMW_0(k)=\HMW^{0,0}(k,\Z)\to \HMW^{0,0}(X,\Z).
\]
\item The homomorphism $\HMW^{*,*}(X,\Z)\to \H^{*,*}(X,\Z)$ is a graded ring homomorphism.
\end{enumerate}

\subsection{A projection formula}\label{sec:projection}

In this section, we prove a projection formula for finite surjective morphisms having trivial relative bundles. Let then $f:X\to Y$ be a finite surjective morphism between smooth connected schemes, and let $\chi:\OO_X\to \omega_f=\omega_{X/k}\otimes f^*\omega_{Y/k}^\vee$ be a fixed isomorphism. Recall from \cite[Example 4.17]{Calmes14b} that we have then a finite MW-correspondence $\alpha:=\alpha(f,\chi):Y\to X$ defined as the composite
\[
\sKMW_0(X)\simeq \sKMW_0(X,\omega_f)\stackrel {f_*}{\to} \chst {d_X}{\Gamma(X)}{Y\times X}{\omega_{Y\times X/k}\otimes \omega_{Y/k}^\vee}\simeq \chst {d_X}{\Gamma(X)}{Y\times X}{\omega_X}
\]
where the first isomorphism is induced by $\chi$, the second homomorphism is the push-forward along the (transpose of the) graph $\Gamma_f:X\to Y\times X$ and the third isomorphism is deduced from the isomorphisms of line bundles
\[
\omega_{Y\times X/k}\otimes \omega_{Y/k}^\vee\simeq \omega_Y\otimes \omega_X\otimes \omega_{Y/k}^\vee\simeq \omega_X\otimes \omega_Y\otimes  \omega_{Y/k}^\vee\simeq \omega_X
\]
where the second isomorphism is $(-1)^{d_Xd_Y}$-times the switch isomorphism.

We observe that $\alpha$ induces a "push-forward" homomorphism $F(X)\to F(Y)$ for any $F\in \psh k$ through the composite  
\[
F(X)=\mathrm{Hom}_{\psh k}(\tZX X,F)\stackrel{(\_)\circ \alpha}{\to} \mathrm{Hom}_{\psh k}(\tZX Y,F)=F(Y).
\]
In particular, we obtain homomorphisms
\[
f_*:\HMW^{p,q}(X,\Z)\to \HMW^{p,q}(Y,\Z)
\]
for any $p,q\in\Z$, which depend on the choice of $\chi$.

On the other hand, $f$ induces a finite MW-correspondence $X\to Y$ that we still denote by $f$ and therefore a pull-back homomorphism
\[
f^*:\HMW^{p,q}(Y,\Z)\to \HMW^{p,q}(X,\Z).
\]
We will need the following lemma to prove the projection formula.

\begin{lem}\label{lem:twoways}
Let $f:X\to Y$ be a finite surjective morphism between smooth connected schemes, and let $\chi:\OO_X\to \omega_f$ be an isomorphism. Let $\Delta_X$ (resp. $\Delta_Y$) be the diagonal embedding $X\to X\times X$ (resp. $Y\to Y\times Y$). Then, the following diagram commutes
\[
\xymatrix{Y\ar[rr]^-{\Delta_Y}\ar@{=}[d] & & Y\times Y\ar[r]^-{(1\times \alpha)} & Y\times X\ar@{=}[d] \\
Y\ar[r]_-{\alpha} & X\ar[r]_-{\Delta_X} & X\times X\ar[r]_-{f\times 1} & Y\times X,}
\]
i.e. $(1\times \alpha)\Delta_Y=(f\times 1)\Delta_X\alpha$.
\end{lem}

\begin{proof}
It suffices to compute both compositions, and we start with the top one. The composite of these two finite MW-correspondences is given by the commutative diagram
\[
\xymatrix{X\ar[r]\ar[d]_-{\Gamma_{(\Delta_Y\circ f)}^t}\ar[r]^-{\Gamma_{(\Delta_Y\circ f)}^t} & Y\times Y\times X\ar[r]\ar[d]^-{1\times \Gamma_{(1\times f)}^t} &Y\times X\ar[d]^-{\Gamma_{(1\times f)}^t} & \\
Y\times Y\times X\ar[r]^-{\Gamma_{\Delta_Y}\times 1\times 1}\ar[d] & Y\times Y\times Y\times Y\times X\ar[r]\ar[d] & Y\times Y\times Y\times X\ar[r]\ar[d] & Y\times X \\
Y\ar[r]_-{\Gamma_{\Delta_Y}} & Y\times Y\times Y\ar[r] \ar[d]& Y\times Y & \\
& Y & & }
\]
where the squares are Cartesian and the non-labelled arrows are projections (vertically to the first factors and horizontally to the last factors). The composite is given by the push-forward along the projection $p:Y\times Y\times Y\times Y\times X\to Y\times Y\times X$ defined by $(y_1,y_2,y_3,y_4,x)\mapsto (y_1,y_4,x)$ of the product of the respective pull-backs to $Y\times Y\times Y\times Y\times X$ of $(\Gamma_{\Delta_Y})_*(\langle 1\rangle)$ and $(\Gamma_{(1\times f)}^t)_*(\langle 1\rangle)$. Using the base change formula (\cite[Proposition 3.2]{Calmes14b}), we see that it amounts to push-forward the product
\[
(\Gamma_{\Delta_Y}\times 1\times 1)_*(\langle 1\rangle)\cdot (1\times \Gamma_{(1\times f)}^t)_*(\langle 1\rangle).
\]
Using the projection formula for Chow-Witt groups with supports (\cite[Corollary 3.5]{Calmes14b}), the latter equals
\[
(\Gamma_{\Delta_Y}\times 1\times 1)_*((\Gamma_{\Delta_Y}\times 1\times 1)^*((1\times \Gamma_{(1\times f)}^t)_*(\langle 1\rangle)))
\]
and the base-change formula once again shows that we have to push-forward along $p$ the cycle
\[
(\Gamma_{\Delta_Y}\times 1\times 1)_*(\Gamma_{(\Delta_Y\circ f)}^t)_*(\langle 1\rangle)
\]
Finally, the equality  $p\circ (\Gamma_{\Delta_Y}\times 1\times 1)=\id$ shows that the composite $(1\times \alpha)\circ \Delta_Y$ is given by the correspondence $(\Gamma_{(\Delta_Y\circ f)}^t)_*(\langle 1\rangle)$. 

For the second composite, we consider the following commutative diagram 
\[
\xymatrix{X\ar[r]^-{\Gamma_f^t}\ar[d]_-{\Gamma_{((f\times 1)\Delta_X)}} & Y\times X\ar[r]\ar[d]^-{1\times \Gamma_{((f\times 1)\Delta_X)}} & X\ar[d]^-{\Gamma_{((f\times 1)\Delta_X)}} & \\
X\times Y\times X\ar[r]^-{\Gamma_f^t\times 1\times 1}\ar[d] & Y\times X\times Y\times X\ar[r]\ar[d] & X\times Y\times X\ar[r]\ar[d] & Y\times X \\
X\ar[r]_{\Gamma_f^t} & Y\times X\ar[d]\ar[r] & X &  \\
& Y & & }
\]
where, as before, the squares are Cartesian and the non-labelled arrows are projections (vertically to the first factors and horizontally to the last factors). Arguing as above, we find that the composite is the push-forward along the projection $q:Y\times X\times Y\times X\to Y\times Y\times X$ omitting the second factor of the product
\[
(\Gamma_f^t\times 1\times 1)_*(\langle 1\rangle)\cdot (1\times \Gamma_{((f\times 1)\Delta_X)})_*(\langle 1\rangle).
\]
The projection and the base-change formulas show that the latter is equal to 
\[
(\Gamma_f^t\times 1\times 1)_*(\Gamma_{((f\times 1)\Delta_X)})_*(\langle 1\rangle)
\]
whose push-forward along $q$ is $(\Gamma_{(\Delta_Y\circ f)}^t)_*(\langle 1\rangle)$ as 
\[
q(\Gamma_f^t\times 1\times 1)(\Gamma_{((f\times 1)\Delta_X)})=\Gamma_{(\Delta_Y\circ f)}^t.
\]
\end{proof}

\begin{thm}[Projection formula]\label{thm:projection}
Let $f:X\to Y$ be a finite surjective morphism between smooth connected schemes, and let $\chi:\OO_X\to \omega_f$ be an isomorphism. For any $x\in \HMW^{p,q}(X,\Z)$ and $y\in \HMW^{p^\prime,q^\prime}(Y,\Z)$, we have 
\[
y\cdot f_*(x)=f_*(f^*y\cdot x)
\]
in $\HMW^{p+p^\prime,q+q^\prime}(Y,\Z)$.
\end{thm}

\begin{proof}
Let $\widetilde{\mathrm{DM}}^{\mathrm{eff}}(k)$ be the category of MW-motives (\cite[\S 3.2]{Deglise16}). By \cite[Corollary 3.3.8]{Deglise16}, we have $\HMW^{p,q}(X,\Z)=\mathrm{Hom}_{\widetilde{\mathrm{DM}}^{\mathrm{eff}}(k)}(\tilde M(X),\tZ\{q\}[p-q])$ for any $p,q\in\Z$. The product structure on MW-motivic cohomology is obtained via the tensor product as follows. If $x$ and $x^\prime$ are respectively in $\mathrm{Hom}_{\widetilde{\mathrm{DM}}^{\mathrm{eff}}(k)}(\tilde M(X),\tZ\{q\}[p-q])$ and $\mathrm{Hom}_{\widetilde{\mathrm{DM}}^{\mathrm{eff}}(k)}(\tilde M(X),\tZ\{q^\prime\}[p^\prime-q^\prime])$, we can take their tensor product to get a morphism $x\otimes x^\prime$ in $\mathrm{Hom}_{\widetilde{\mathrm{DM}}^{\mathrm{eff}}(k)}(\tilde M(X)\otimes \tilde M(X),\tZ\{q\}\otimes \tZ\{q^\prime\}[p+p^\prime-q-q^\prime])$. Now, $\tilde M(X)\otimes \tilde M(X)=\tilde M(X\times X)$ and the diagonal morphism $\Delta_X:X\to X\times X$ induces a morphism $\tilde M(X)\to \tilde M(X\times X)$. Composing the latter with $x\otimes x^\prime$, we obtain a morphism $x\cdot x^\prime\in \mathrm{Hom}_{\widetilde{\mathrm{DM}}^{\mathrm{eff}}(k)}(\tilde M(X),\tZ\{q\}\otimes \tZ\{q^\prime\}[p+p^\prime-q-q^\prime])$ which represent the product of $x$ and $x^\prime$ (after identification of $\tZ\{q\}\otimes \tZ\{q^\prime\}$ with $\tZ\{q+q^\prime\}$). 

This being said, let then $x\in \mathrm{Hom}_{\widetilde{\mathrm{DM}}^{\mathrm{eff}}(k)}(\tilde M(X),\tZ\{q\}[p-q])$ and let $y\in \mathrm{Hom}_{\widetilde{\mathrm{DM}}^{\mathrm{eff}}(k)}(\tilde M(Y),\tZ\{q^\prime\}[p^\prime-q^\prime])$. The product $y\cdot f_*(x)$ is then of the form $(y\otimes x)\circ (1\otimes \alpha)\circ \Delta_Y$, while $f_*(f^*y\cdot x)$ is of the form $(y\otimes x)\circ (f\otimes 1)\circ \Delta_X\circ \alpha$. The result then follows from Lemma \ref{lem:twoways}.
\end{proof}

\begin{rem}
It would suffice to have a fixed isomorphism $\L\otimes \L\simeq \omega_f$ (for some line bundle $\L$ on $X$) to get an orientation in the sense of \cite[\S 2.2]{Calmes14b} and thus a finite MW-correspondence $\alpha$ as above. We let the reader make the necessary modifications in the arguments of both Lemma \ref{lem:twoways} and Theorem \ref{thm:projection}.
\end{rem}

\begin{rem}
It follows from \cite[Theorem 3.4.3]{Deglise16} that the same formula holds for the left module structure, i.e.
\[
f_*(x)\cdot y =f_*(x\cdot f^*y).
\]
\end{rem}

\begin{exem}\label{ex:inseparable}
As usual, it follows from the projection formula that the composite $f_*f^*$ is multiplication by $f_*(\langle 1\rangle)$. Let us now compute $f^*f_*$ in some situations that will be used later. Let us start with the general situation, i.e. $f:X\to Y$ is a finite surjective morphism and $\chi:\OO_X\to \omega_f$ an isomorphism. The composite $f^*f_*$ is given by precomposition with the correspondence $f\circ\alpha(f,\chi)$ which we can compute using the diagram
\[
\xymatrix{X\times_Y X\ar[r]^-{(1\times 1)}\ar[d]_-{(1\times 1)} & X\times X\ar[r]\ar[d]_-{1\times \Gamma_f^t} & X\ar[d]^-{\Gamma_f^t} & \\
X\times X\ar[r]_-{\Gamma_f\times 1}\ar[d] & X\times Y\times X\ar[r]\ar[d] & Y\times X\ar[r]\ar[d] & X \\
X\ar[r]_-{\Gamma_f} & X\times Y\ar[d]\ar[r] & Y &  \\
 & X & & }
\] 
where the non-labelled vertical arrows are projections on the first factor and the non-labelled horizontal arrows are projections on the second factor. As usual, the base change formula shows that the composite is equal to the projection on $X\times X$ of 
\[
(\Gamma_f\times 1)_*(\langle 1\rangle)\cdot (1\times \Gamma_f^t)_*(\langle 1\rangle).
\] 
In general the top left square is not transverse, and we can't use the base-change formula to compute the above product. 

Suppose now that $f:X\to Y$ is finite and \'etale. In that case, we have a canonical isomorphism $f^*\omega_Y\simeq \omega_X$ yielding a canonical choice for the isomorphism
\[
\chi:\OO_X\to \omega_f.
\] 
Moreover, $X\times_Y X$ decomposes as $X\times_Y X=X_1\sqcup X_2\sqcup \ldots \sqcup X_n$ where each term $X_i$ is finite and \'etale over $X$ with "structural" morphism $p_i:X_i\to X$. In that case, the above top right square is transverse and we see that 
\[
(\Gamma_f\times 1)_*(\langle 1\rangle)\cdot (1\times \Gamma_f^t)_*(\langle 1\rangle)=(\Gamma_f\times 1)_*(\Delta_*\sum (p_i)_*(\langle 1\rangle)).
\] 
where $\Delta:X\to X\times X$ is the diagonal map. Thus the composite $f\circ\alpha(f,\chi)$ is equal to $\Delta_*\sum (p_i)_*(\langle 1\rangle)$. It follows immediately that we have a commutative diagram
\[
\xymatrix{\HMW^{p,q}(X,\Z)\ar[r]^-{\sum p_i^*}\ar[d]_-{f_*} & \bigoplus_i \HMW^{p,q}(X_i,\Z) \ar[d]^-{\sum (p_i)_*} \\
\HMW^{p,q}(Y,\Z)\ar[r]_-{f^*} & \HMW^{p,q}(X,\Z)}
\]
for any $p,q\in \Z$. 

Suppose next that $\mathrm{char}(k)=p$, that $X\subset Y\times \A^1$ is the set of zeroes of $t^p-a$ for some global section $a\in \OO_Y(Y)$ (we still suppose that $X$ is smooth over $k$). In that case, we see that the reduced scheme of $X\times_Y X$ is just $X$ (but the former has nilpotent elements) and it follows that $f\circ\alpha(f,\chi)$ is a correspondence supported on the diagonal $\Delta(X)\subset X\times X$. It follows that there is an element $\sigma\in \KMW_0(X)$ such that the following diagram commutes
\[
\xymatrix{\sKMW_0(X)\ar[r]\ar[d]_-{\cdot \sigma} & \cor k(X,X)\ar[d]^-{f\circ\alpha(f,\chi)}\\
\sKMW_0(X)\ar[r] & \cor k(X,X)}
\]
where the horizontal arrows are induced by the push-forward map $\Delta_*:\sKMW_0(X)\to \chst {d_X}{\Delta(X)}{X\times X}{\omega_X}$. Now, $\sigma$ can be computed using the composite $\KMW_0(k(X))\to \KMW_0(k(Y))\to \KMW_0(k(X))$, where the first map is the push-forward (defined using $\chi$) and the second map the pull-back. It follows essentially from \cite[Lemme 6.4.6]{Fasel08a} that $\sigma=p_{\epsilon}$.
\end{exem}

\subsection{The homomorphism}\label{sec:homomorphism}

Let $L/k$ be a finitely generated field extension. It follows from the definition of MW-motivic cohomology that $\HMW^{p,q}(L,\Z)=0$ provided $p>q$. The next step is then to identify $\HMW^{p,p}(L,\Z)$. To this aim,  we constructed in \cite[Theorem 4.2.2]{Deglise16} a graded ring homomorphism
\[
\KMW_*(L)\to \bigoplus_{n\in\Z}\HMW^{n,n}(L,\Z) 
\]
which we now recall. For $a\in L^\times$, we can consider the corresponding morphism $a:\spec L\to \gmk k$ which defines a finite MW-correspondence $\Gamma_a$ in $\cor k(L,\gmk k)$. Now, we have a surjective homomorphism $\cor k(L,\gmk k)\to \HMW^{1,1}(L,\Z)$ and we let $s([a])$ be the image of $\Gamma_a$ under this map. Next, consider the element 
\[
\eta[t]\in \sKMW_0(\gmk L)=\cor k(\gmk L,k)=\cor k(\gmk k\times L,k)=\uHom (\tZX {\gmk k},\tZX k)(L).
\] 
We define $s(\eta)$ to be its image under the projections
\[
\uHom (\tZX {\gmk k},\tZX k)(L)\to \uHom (\tZ \{1\},\tZX k)(L)\to \HMW^{-1,-1}(L,\Z).
\]
The following theorem is proved in \cite[Theorem 4.2.2]{Deglise16} (using computations of \cite[\S 6]{Fasel16b}).

\begin{thm}
The associations $[a]\mapsto s([a])$ and $\eta\mapsto s(\eta)$ induce a homomorphism of graded rings
\[
\Phi_L:\KMW_*(L)\to \bigoplus_{n\in\Z}\HMW^{n,n}(L,\Z).
\]
\end{thm}

By construction, the above homomorphism fits in a commutative diagram of graded rings
\[
\xymatrix{\KMW_*(L)\ar[r]^-{\Phi_L}\ar[d] & \bigoplus_{n\in\Z}\HMW^{n,n}(L,\Z)\ar[d] \\
\KM_*(L)\ar[r] & \bigoplus_{n\in\Z}\H^{n,n}(L,\Z)}
\]
where the vertical projections are respectively the natural map from Milnor-Witt $K$-theory to Milnor $K$-theory and the ring homomorphism of the previous section, and the bottom horizontal homomorphism is the map constructed by Totaro-Nesterenko-Suslin.


\section{Main theorem}\label{sec:main}

\subsection{A left inverse}\label{sec:left}

In this section, we construct for $q\geq 0$ a left inverse to the homomorphism $\Phi_L$ of Section \ref{sec:homomorphism}. By definition, 
\[
\tZX {\gm^q}(L):=\bigoplus_{x\in (\gmk L^q)^{(q)}}\chst qx{\gmk L^q}{\omega_{\gm^q}}.
\] 
Now, for any point $x$ in $(\gmk L^q)^{(q)}$ with maximal ideal $\mathfrak m$, we have an exact sequence
\[
\mathfrak m/\mathfrak m^2\to \Omega_{\gmk L^q/k}\to \Omega_{L(x)/k}\to 0.
\]
Using the fact that $k$ is perfect and counting dimensions, we see that this sequence is also exact on the left. We find an isomorphism
\[
\wedge^q (\mathfrak m/\mathfrak m^2)^\vee\otimes \omega_{\gmk L^q/k}\simeq \omega_{L(x)/k}
\]
Now, $\omega_{\gmk L^q/k}\simeq p_1^*\omega_{\gm^q/k}\otimes p_2^*\omega_{L/k}$ and it follows that 
\[
\wedge^q (\mathfrak m/\mathfrak m^2)^\vee\otimes \omega_{\gm^q}\simeq \omega_{L(x)/k}\otimes_{L}\omega_{L/k}^\vee
\]
yielding
\[
\tZX {\gm^q}(L)=\bigoplus_{x\in (\gmk L^q)^{(q)}} \KMW_0(L(x),\omega_{L(x)/k}\otimes_{L}\omega_{L/k}^\vee).
\]
Now any closed point $x$ in $(\gmk L^q)^{(q)}$ can be identified with a $q$-uple $(x_1,\ldots,x_q)$ of elements of $L(x)$. For any such $x$, we define a homomorphism
\[
f_{x}:\KMW_0(L(x),\omega_{L(x)/k}\otimes_{L}\omega_{L/k}^\vee)\to \KMW_q(L)
\]
by $f_x(\alpha)=\Tr_{L(x)/L}(\alpha\cdot [x_1,\ldots,x_q])$. We then obtain a homomorphism
\[
f:\tZX {\gm^q}(L)\to \KMW_q(L)
\]
which is easily seen to factor through $(\tZ\{q\})(L)$ since $[1]=0\in \KMW_1(L)$. 

We now check that this homomorphism vanishes on the image of $(\tZ\{q\})(\A^1_L)$ in $(\tZ\{q\})(L)$ under the boundary homomorphism. This will follow from the next lemma.

\begin{lem}
Let $Z\in \Adm(\A^1_L,\gm^q)$. Let moreover $p:\gmk L^q\to \spec L$ and $p_{\A^1_L}:\A^1_L\times \gm^q\to \A^1_L$ be the projections and $Z_i:=p_{\A^1_L}^{-1}(i)\cap Z$ (endowed with its reduced structure) for $i=0,1$. Let $j_i:\spec L\to \A^1_L$ be the inclusions in $i=0,1$ and let $g_i:\gmk L^q\to \A^1_L\times \gm^q$ be the induced maps.  Then the homomorphisms
\[
p_*(g_i)^*:\chst qZ{\A^1_L\times \gm^q}{\omega_{\gm^q}}\to \chst q{Z_i}{\gmk L^q}{\omega_{\gm^q}}\to \KMW_0(L)
\]
are equal.
\end{lem}

\begin{proof}
For $i=0,1$, consider the Cartesian square
\[
\xymatrix{\gmk L^q\ar[r]^-{g_i}\ar[d]_-{p} & \A^1_L\times \gmk L^q\ar[d]^-{p_{\A^1_L}}\\ 
\spec L\ar[r]_-{j_i} & \A^1_L} 
\]
We have $(j_i)^*(p_{\A^1_L})_*=p_*(g_i)^*$ by base change. The claim follows from the fact that $(j_0)^*=(j_1)^*$ by homotopy invariance.
\end{proof}

\begin{prop}\label{prop:left}
The homomorphism $f:\tZX {\gm^q}(L)\to \KMW_q(L)$ induces a homomorphism
\[
\theta_L:\HMW^{q,q}(L,\Z)\to \KMW_q(L)
\]
for any $q\geq 1$. 
\end{prop}

\begin{proof}
Observe that the group $\HMW^{q,q}(L,\Z)$ is the cokernel of the homomorphism
\[
\partial_0-\partial_1:\tZ\{q\}(\A^1_L)\to \tZ\{q\}(L)
\]
It follows from \cite[Example 4.16]{Calmes14b} that $\partial_i:\tZX {\gm^q}(\A^1_L)\to \tZX {\gm^q}(L)$ is induced by $g_i^*$. We can use the above lemma to conclude.
\end{proof}

\begin{coro}\label{cor:splitinjective}
The homomorphism 
\[
\Phi_L:\bigoplus_{n\in\Z} \KMW_n(L) \to \bigoplus_{n\in\Z} \HMW^{n,n}(L,\Z).
\]
is split injective.
\end{coro}

\begin{proof}
It suffices to check that $\theta_L\Phi_L=\id$, which is straightforward.
\end{proof}

The following result will play a role in the proof of the main theorem.

\begin{prop}\label{prop:transe}
Let $n\in\Z$ and let $F/L$ be a finite field extension. Then, the following diagram commutes
\[
\xymatrix{\HMW^{n,n}(F,\Z)\ar[d]_-{\Tr_{F/L}}\ar[r] ^-{\theta_F }& \KMW_n(F)\ar[d]^-{\Tr_{F/L}} \\
\HMW^{n,n}(L,\Z)\ar[r]_-{\theta_L} & \KMW_n(L). }
\]
\end{prop}

\begin{proof}
Let $X$ be a smooth connected scheme and let $\beta\in \cor k(X,\gm^{\times n})$ be a finite MW-correspondence with support $T$ (see \cite[Definition 4.7]{Calmes14b} for the notion of support). Each connected component $T_i$ of $T$ has a fraction field $k(T_i)$ which is a finite extension of $k(X)$ and, arguing as in the beginning of Section \ref{sec:left}, we find that $\beta$ can be seen as an element of 
\[
\bigoplus_{i} \KMW_0(k(T_i),\omega_{k(T_i)/k}\otimes \omega_{k(X)/k}^\vee)
\]
Now, the morphism $T_i\subset X\times \gm^{\times n}\to \gm^{\times n}$ gives invertible global sections $a_1,\ldots,a_n$ and we define a map
\[
\theta_X:\cor k(X,\gm^{\times n})\to \KMW_n(k(X))
\]
by $\beta\mapsto \sum \Tr_{k(T_i)/k(X)}(\beta_i[a_1,\ldots,a_n])$, where $\beta_i$ is the component of $\beta$ in the group $\KMW_0(k(T_i),\omega_{k(T_i)/k(X)})$. This map is easily seen to be a homomorphism, and its limit at $k(X)$ is the morphism defined at the beginning of Section \ref{sec:left}.

Let now $X$ and $Y$ be smooth connected schemes over $k$, $f:X\to Y$ be a finite morphism and $\chi:\OO_X\to \omega_f$ be an isomorphism inducing a finite MW-correspondence $\alpha(f,\chi):Y\to X$ as in Section \ref{sec:projection}. We claim that the diagram
\[
\xymatrix{\cor k(X,\gm^{\times n})\ar[r]^-{\theta_X}\ar[d]_-{\circ\alpha(f,\chi)} & \KMW_n(k(X))\ar[d]^-{\Tr_{k(X)/k(Y)}} \\
\cor k(Y,\gm^{\times n})\ar[r]_-{\theta_Y} & \KMW_n(k(Y)),}
\]
where the right arrow is obtained using $\chi$, commutes. If $\beta$ is as above, we have 
\[
\Tr_{k(X)/k(Y)}(\theta_X(\beta))=\sum \Tr_{k(X)/k(Y)}( \Tr_{k(T_i)/k(X)}(\beta_i[a_1,\ldots,a_n]))
\]
and the latter is equal to $\sum  \Tr_{k(T_i)/k(Y)}(\beta_i[a_1,\ldots,a_n])$ by functoriality of the transfers. On the other hand, the isomorphism $\chi:\OO_X\to \omega_f$ can be seen as an element in $\sKMW_0(X,\omega_f)$, yielding an element of $\KMW_0(k(X),\omega_f)$ that we still denote by $\chi$. The image of $\beta\circ\alpha(f,\chi)$ can be seen as the element $\beta\cdot \chi$ of   
\[
\bigoplus_{i} \KMW_0(k(T_i),\omega_{k(T_i)/k(Y)})
\] 
where we have used the isomorphism 
\[
\omega_{k(T_i)/k(X)}\otimes \omega_f= \omega_{k(T_i)/k(X)}\otimes \omega_{k(X)/k(Y)}\simeq \omega_{k(T_i)/k(Y)}.
\]
It is now clear that $\theta_Y(\beta\circ\alpha(f,\chi))=\sum  \Tr_{k(T_i)/k(Y)}(\beta_i[a_1,\ldots,a_n])$ and the result follows. 
\end{proof}


\subsection{Proof of the main theorem}

In this section we prove our main theorem, namely that the homomorphism 
\[
\Phi_L:\bigoplus_{n\in\Z} \KMW_n(L) \to \bigoplus_{n\in\Z} \HMW^{n,n}(L,\Z)
\]
is an isomorphism. We first observe that $\Phi_L$ is an isomorphism in degrees $\leq 0$. In degree $0$, we indeed know from \cite[\S 6]{Calmes14b} that both sides are $\KMW_0(L)$. Next, \cite[Lemma 6.0.1]{Fasel16b} yields
\[
\Phi_L(\langle a\rangle)=\Phi_L(1+\eta [a])=1+s(\eta)s(a)=\langle a\rangle.
\]
It follows $\Phi_L$ is a homomorphism of graded $\KMW_0(L)$-algebras and the result in degrees $<0$ follows then from the fact that $\HMW^{p,p}(L,\Z)=W(L)=\KMW_{-p}(L)$ by \cite[\S 6]{Calmes14b} and \cite[Proposition 4.1.2]{Deglise16}.

We now prove the result in positive degrees, starting with $n=1$. Recall that we know from Corollary \ref{cor:splitinjective} that $\Phi_L$ is split injective, and that it therefore suffices to prove that it is surjective to conclude.

For any $d,n\geq 1$ and any field extension $L/k$ let $M_n^{(d)}(L)\subset \cor k(L,\gm^{\times n})$ be the subgroup of correspondences whose support is a finite union of field extensions $E/L$ of degree $\leq d$ (see \cite[Definition 4.7]{Calmes14b} for the notion of support of a correspondence). Let $\HMW^{n,n}(L,\Z)^{(d)}\subset \HMW^{n,n}(L,\Z)$ be the image of $M_n^{(d)}(L)$ under the surjective homomorphism
\[
\cor k(L,\gm^{\times n})\to \HMW^{n,n}(L,\Z).
\]
Observe that 
\[
\HMW^{n,n}(L,\Z)^{(d)}\subset \HMW^{n,n}(L,\Z)^{(d+1)} \quad \text{and} \quad \HMW^{n,n}(L,\Z)=\cup_{d\in \N}\HMW^{n,n}(L,\Z)^{(d)}.
\]
\begin{lem}
The subgroup $\HMW^{n,n}(L,\Z)^{(1)}\subset \HMW^{n,n}(L,\Z)$ is the image of the homomorphism
\[
\Phi_L:\KMW_n(L)\to \HMW^{n,n}(L,\Z).
\]
\end{lem}

\begin{proof}
By definition, observe that the homomorphism $\KMW_n(L)\to \HMW^{n,n}(L,\Z)$ factors through $\HMW^{n,n}(L,\Z)^{(1)}$. Let then $\alpha\in \HMW^{n,n}(L,\Z)^{(1)}$. We may suppose that $\alpha$ is the image under the homomorphism $\cor k(L,\gm^{\times n})\to \HMW^{n,n}(L,\Z)$ of a correspondence $a$ supported on a field extension $E/L$ of degree $1$, i.e. $E=L$. It follows that $a$ is determined by a form $\phi\in \KMW_0(L)$ and a $n$-uple $a_1,\ldots,a_n$ of elements of $L$. This is precisely the image of $\Phi_L(\phi\cdot [a_1,\ldots,a_n])$ under the homomorphism $\KMW_n(L)\to \HMW^{n,n}(L,\Z)$.
\end{proof}

\begin{prop}\label{prop:inductionstep}
For any $d\geq 2$, we have $\HMW^{1,1}(L,\Z)^{(d)}\subset \HMW^{1,1}(L,\Z)^{(d-1)}$.
\end{prop}

\begin{proof}
By definition, $\HMW^{1,1}(L,\Z)^{(d)}$ is generated by correspondences whose supports are field extensions $E/L$ of degree at most $d$. Such correspondences are determined by an element $a\in E^\times$ given by the composite $\spec E\to \gmk L\to \gm$ together with a form $\phi\in \KMW_0(E,\omega_{E/L})$ given by the isomorphism 
\[
\KMW_0(E,\omega_{E/L})\to \chst 1{\spec E}{\gmk L}{\omega_{\gmk L/L}}.
\]
We denote this correspondence by the pair $(a,\phi)$. Recall from \cite[Lemma 2.4]{Calmes14b} that there is a canonical orientation $\xi$ of $\omega_{E/L}$ and thus a canonical element $\chi\in \cor k(\spec L,\spec E)$ yielding the transfer map 
\[
\Tr_{E/L}:\cor k(\spec E,\gm)\to \cor k(\spec L,\gm)
\]
which is just the composition with $\chi$ (\cite[Example 4.17]{Calmes14b}). Now $\phi=\psi\cdot \xi$ for $\psi\in \KMW_0(E)$, and it is straightforward to check that the Chow-Witt correspondence $(a,\psi)$ in $\cor k(\spec E,\gm)$ determined by $a\in E^\times$ and $\psi\in \KMW_0(E)$ satisfies $\Tr_{E/L}(a,\psi)=(a,\phi)$. Now $(a,\psi)\in \HMW^{1,1}(E,\Z)^{(1)}$ and therefore belongs to the image of the homomorphism $\KMW_1(E)\to \HMW^{1,1}(E,\Z)$. There exists thus $a_1,\ldots,a_n,b_1,\ldots,b_m\in E^\times$ (possibly equal) such that $(a,\psi)=\sum s(a_i)-\sum s(b_j)$. To prove the lemma, it suffices then to show that $\Tr_{E/L}(s(b))\in \HMW^{1,1}(L,\Z)^{(d-1)}$ for any $b\in E^\times$. 

Let thus $b\in E^\times$. By definition, $s(b)\in \H^{1,1}(E,\tZ)$ is the class of the correspondence $\tilde\gamma(b)$ associated to the morphism of schemes $\spec E\to \gm$ corresponding to $b$. If $F(b)\subset E$ is a proper subfield, we see that $\Tr_{E/L}(s(b))\in \H^{1,1}(F,\tZ)^{(d-1)}$, and we may thus suppose that the minimal polynomial $p$ of $b$ over $F$ is of degree $d$. By definition, $\Tr_{E/L}(s(b))$ is then represented by the correspondence associated to the pair $(b,\langle 1\rangle\cdot \xi)$. Consider the total residue homomorphism (twisted by the vector space $\omega_{F[t]/k}\otimes \omega_{F/k}^\vee$)
\begin{equation}\label{eq:residue}
\partial:\KMW_1(F(t),\omega_{F(t)/k}\otimes \omega_{F/k}^\vee)\to \bigoplus_{x\in \gmk F^{(1)}} \KMW_0(F(x),(\mathfrak m_x/\mathfrak m_x^2)^\vee \otimes _{F[t]} \omega_{F[t]/k}\otimes \omega_{F/k}^\vee)
\end{equation}
where $\mathfrak m_x$ is the maximal ideal corresponding to $x$. Before working further with this homomorphism, we first identify $(\mathfrak m_x/\mathfrak m_x^2)^\vee \otimes _{F[t]} \omega_{F[t]/k}\otimes \omega_{F/k}^\vee$. Consider the canonical exact sequence of $F(x)$-vector spaces
\[
\mathfrak m_x/\mathfrak m_x^2\to \Omega_{F[t]/k}\otimes_{F[t]} F(x) \to \Omega_{F(x)/k}\to 0.
\]
A comparison of the dimensions shows that the sequence is also left exact (use the fact that $F(x)$ is the localization of a smooth scheme of dimension $d$ over the perfect field $k$), and we thus get a canonical isomorphism
\[
\omega_{F[t]/k}\otimes_{F[t]} F(x)\simeq \mathfrak m_x/\mathfrak m_x^2 \otimes_{F(x)}\omega_{F(x)/k}.
\]
It follows that
\[
(\mathfrak m_x/\mathfrak m_x^2)^\vee \otimes _{F[t]} \omega_{F[t]/k}\otimes \omega_{F/k}^\vee\simeq \omega_{F(x)/k}\otimes_F \omega_{F/k}^\vee.
\] 
We can thus rewrite the residue homomorphism \eqref{eq:residue} as a homomorphism
\[
\partial:\KMW_1(F(t),\omega_{F(t)/k}\otimes \omega_{F/k}^\vee)\to \bigoplus_{x\in (\A^1_F\setminus 0)^{(1)}} \KMW_0(F(x),\omega_{F(x)/k}\otimes_F \omega_{F/k}^\vee)
\]
Moreover, an easy dimension count shows that the canonical exact sequence
\[
\Omega_{F/k}\otimes F[t]\to \Omega_{F[t]/k}\to \Omega_{F[t]/F}\to 0
\]
is also exact on the left, yielding a canonical isomorphism $\omega_{F(t)/k}\simeq \omega_{F/k}\otimes \omega_{F(t)/F}$ and thus a canonical isomorphism $\omega_{F/k}^\vee\otimes \omega_{F(t)/k}\simeq \omega_{F(t)/F}$.
If $n$ is the transcendance degree of $F$ over $k$, we see that the canonical isomorphism 
\[
\omega_{F/k}^\vee\otimes \omega_{F(t)/k}\simeq\omega_{F(t)/k}\otimes \omega_{F/k}^\vee
\] 
is equal to $(-1)^{n(n+1)}$-times the switch isomorphism, i.e. is equal to the switch isomorphism. Altogether, the residue homomorphism reads as 
\[
\partial:\KMW_1(F(t),\omega_{F(t)/F})\to \bigoplus_{x\in (\A^1_F\setminus 0)^{(1)}} \KMW_0(F(x),\omega_{F(x)/k}\otimes_F \omega_{F/k}^\vee).
\]

Let now $p(t)\in F[t]$ be the minimal polynomial of $b$ over $F$.
Following \cite[Definition 4.26]{Morel08} (or \cite[\S 2]{Calmes14b}), write $p(t)=p_0(t^{l^m})$ with $p_0$ separable and set $\omega=p_0^\prime(t^{l^m})\in F[t]$ if $\mathrm{char}(k)=l$. If $\mathrm{char}(k)=0$, set $\omega=p^\prime(t)$. It is easy to see that the element $\langle \omega\rangle [p]\cdot dt$ of $\KMW_1(F(t),\omega_{F(t)/F})$ ramifies in $b\in \gmk F^{(1)}$ and on (possibly) other points corresponding to field extensions of degree $\leq d-1$. Moreover, the residue at $b$ is exactly $\langle 1\rangle \cdot \xi$, where $\xi$ is the canonical orientation of $\omega_{F(b)/F}$.

Write the minimal polynomial $p(t)\in F[t]$ of $b$ as $p=\sum_{i=0}^d \lambda_i t^i$ with $\lambda_d=1$ and $\lambda_0\in F^\times$, and decompose $\omega=c\prod_{j=1}^n q_j^{m_j}$, where $c\in F^\times$ and $q_j\in F[t]$ are irreducible monic polynomials. Let $f=(t-1)^{d-1}(t-(-1)^{d}\lambda_0) \in F[t]$. Observe that $f$ is monic and satisfies $f(0)=p(0)$. Let $F(u,t)=(1-u)p+uf$. Since $f$ and $p$ are monic and have the same constant terms, it follows that $F(u,t)=t^d+\ldots+\lambda_0$ and therefore $F$ defines an element of $\Adm(\A^1_F,\gm)$. For the same reason, every $q_j$ (seen as a polynomial in $F[u,t]$ constant in $u$) defines an element in $\Adm(\A^1_F,\gm)$. The image of $\langle\omega\rangle [F]\cdot dt \in \KMW_1(F(u,t),\omega_{F(u,t)/F(u)})$ under the residue homomorphism
\[
\partial:\KMW_1(F(u,t),\omega_{F(u,t)/F(u)})\to \bigoplus_{x\in (\A^1_F\times_k \gm)^{(1)}} \KMW_0(F(x),(\mathfrak m_x/\mathfrak m_x^2)^\vee\otimes _{F[t]}\omega_{F[u,t]/F[u]})
\]
is supported on the vanishing locus of $F$ and the $g_j$, and it follows that it defines a finite Chow-Witt correspondence $\alpha$ in $\cor k(\A^1_F,\gm)$. The evaluation $\alpha(0)$ of $\alpha$ at $u=0$ consists from $\langle 1\rangle \cdot \xi$ and correspondences supported on the vanishing locus of the $q_j$, while $\alpha(1)$ is supported on the vanishing locus of $f$ and the $q_j$. The class of $\langle 1\rangle \cdot \xi$ is then an element of $\HMW^{1,1}(L,\Z)^{(d-1)}$.
\end{proof}

\begin{coro}\label{cor:degree1}
The homomorphism 
\[
\Phi_L: \KMW_1(L)\to \HMW^{1,1}(F,\Z)
\]
is an isomorphism for any finitely generated field extension $F/k$.
\end{coro}

\begin{proof}
We know that the homomorphism is (split) injective. The above proposition shows that $\HMW^{1,1}(F,\Z)=\HMW^{1,1}(F,\Z)^{(1)}$ and the latter is the image of $\KMW_1(L)$ under $\Phi_L$. It follows that $\Phi_L$ is surjective.
\end{proof}

We can now prove that $\theta$ respects transfers following \cite[Lemma 5.11]{Mazza06} and \cite[Lemma 9.5]{Neshitov14}.

\begin{thm}\label{thm:respect}
Let $n\in\N$ and let $F/L$ be a finite field extension. Then the following diagram commutes
\[
\xymatrix{\KMW_n(F)\ar[r]^-{\Phi_F}\ar[d]_-{\Tr_{F/L}} & \HMW^{n,n}(F,\Z)\ar[d]^-{\Tr_{F/L}} \\
\KMW_n(L)\ar[r]_-{\Phi_L} & \HMW^{n,n}(L,\Z). }
\]
\end{thm}

\begin{proof}
First, we know from Proposition \ref{prop:transe} that the diagram
\[
\xymatrix{\HMW^{n,n}(F,\Z)\ar[d]_-{\Tr_{F/L}}\ar[r] ^-{\theta_F }& \KMW_n(F)\ar[d]^-{\Tr_{F/L}} \\
\HMW^{n,n}(L,\Z)\ar[r]_-{\theta_L} & \KMW_n(L). }
\]
commutes. If $\Phi_F$ and $\Phi_L$ are isomorphisms, it follows from Corollary \ref{cor:splitinjective} that $\theta_F$ and $\theta_L$ are their inverses and thus that the diagram
\[
\xymatrix{\KMW_n(F)\ar[r]^-{\Phi_F}\ar[d]_-{\Tr_{F/L}} & \HMW^{n,n}(F,\Z)\ar[d]^-{\Tr_{F/L}} \\
\KMW_n(L)\ar[r]_-{\Phi_L} & \HMW^{n,n}(L,\Z). }
\]
also commutes. We may then suppose, using Corollary \ref{cor:degree1} that $n\geq 2$. Additionally, we may suppose that $[F:L]=p$ for some prime number $p$. Following \cite[Lemma 5.11]{Mazza06}, we first assume that $L$ has no field extensions of degree prime to $p$. In that case, it follows from \cite[Lemma 3.25]{Morel08} that $\KMW_n(F)$ is generated by elements of the form $\eta^{m}[a_1,a_2,\ldots,a_{n+m}]$ with $a_1\in F^\times$ and $a_i\in L^{\times}$ for $i\geq 2$. We conclude from the projection formula \ref{thm:projection}, its analogue in Milnor-Witt $K$-theory and the fact that $\Phi$ is a ring homomorphism that the result holds in that case.

Let's now go back to the general case, i.e. $[F:L]=p$ without further assumptions. Let $L^\prime$ be the maximal prime-to-$p$ field extension of $L$.  Let $\alpha\in \HMW^{n,n}(L,\Z)$ be such that its pull-back to $\HMW^{n,n}(L^\prime,\Z)$ vanishes. It follows then that there exists a finite field extension $E/L$ of degree $m$ prime to $p$ such that the pull-back of $\alpha$ to $\HMW^{n,n}(E,\Z)$ is trivial. Let $f:\spec E\to \spec L$ be the corresponding morphism. For any unit $b\in E^\times$, we have $\langle b\rangle\cdot f^*(\alpha)=0$ and it follows from the projection formula once again that $f_*(\langle b\rangle\cdot f^*(\alpha))=f_*(\langle b\rangle)\cdot \alpha=0$. We claim that there is a unit $b\in E^\times$ such that $f_*(\langle b\rangle)=m_{\epsilon}$. Indeed, we can consider the factorization $L\subset F^{sep}\subset E$ where $F^{sep}$ is the separable closure of $L$ in $E$ and the extension $F^{sep}\subset E$ is purely inseparable. If the claim holds for each extension, then it holds for $L\subset E$. We may thus suppose that the extension is either separable or purely inseparable. In the first case, the claim follows from \cite[Lemme 2]{Serre68} while the second case follows from \cite[Th\'eor\`eme 6.4.13]{Fasel08a}. Thus, for any $\alpha\in \HMW^{n,n}(L,\Z)$ vanishing in $\HMW^{n,n}(L^\prime,\Z)$ there exists $m$ prime to $l$ such that $m_{\epsilon}\alpha=0$.

Let now $\alpha\in \KMW_n(F)$ and  $t(\alpha)=(\Tr_{F/L}\circ \Phi_F-\Phi_L\circ \Tr_{F/L})(\alpha)\in \HMW^{n,n}(L,\Z)$. Pulling back to $L^\prime$ and using the previous case, we find that $m_{\epsilon}t(\alpha)=0$. On the other hand, the above arguments show that if the pull-back of $t(\alpha)$ to $F$ is trivial, then $p_{\epsilon}t(\alpha)=0$ and thus $t(\alpha)=0$ as $(p,m)=1$. Thus, we are reduced to show that $f^*(t(\alpha))=0$ where $f:\spec F\to \spec L$ is the morphism corresponding to $L\subset F$. 

Suppose first that $F/L$ is purely inseparable. In that case, we know from Example \ref{ex:inseparable} that $f^*f_*:\HMW^{n,n}(F,\Z)\to \HMW^{n,n}(F,\Z)$ is multiplication by $p_{\epsilon}$. The same property holds for Milnor-Witt $K$-theory. This is easily checked using the definition of the transfer, or alternatively using Proposition \ref{prop:transe}, the fact that $\theta_F$ is surjective and Example \ref{ex:inseparable}. Altogether, we see that $f^*t(\alpha)=p_{\epsilon}\Phi_F(\alpha)-\Phi_F(p_\epsilon\cdot \alpha)$ and therefore $f^*t(\alpha)=0$ since $\Phi_F$ is $\KMW_0(F)$-linear.

Suppose next that $F/L$ is separable. In that case, we have $F\otimes_L F=\prod_i F_i$ for field extensions $F_i/F$ of degree $\leq p-1$. We claim that the diagrams
\[
\xymatrix{\KMW_n(F)\ar[d]_-{\Tr_{F/L}}\ar[r] & \oplus_i \KMW_n(F_i)\ar[d]^-{\sum \Tr_{F_i/F}} \\
\KMW_n(L)\ar[r] & \KMW_n(F)}
\]
and 
\[
\xymatrix{\HMW^{n,n}(F,\Z)\ar[d]_-{\Tr_{F/L}}\ar[r] & \oplus_i \HMW^{n,n}(F_i,\Z)\ar[d]^-{\sum \Tr_{F_i/F}} \\
\HMW^{n,n}(L,\Z)\ar[r] & \HMW^{n,n}(F,\Z)}
\]
commute. The second one follows from Example \ref{ex:inseparable} and the first one from \cite[Lemma 9.4]{Neshitov14} (or alternatively from Proposition \ref{prop:transe}, the fact that $\theta_F$ is surjective and Example \ref{ex:inseparable}). By induction, each of the diagrams
\[
\xymatrix{\KMW_n(F_i)\ar[r]^-{\Phi_{F_i}}\ar[d]_-{\Tr_{F_i/F}} & \HMW^{n,n}(F_i,\Z)\ar[d]^-{\Tr_{F_i/F}} \\
\KMW_n(F)\ar[r]_-{\Phi_F} & \HMW^{n,n}(F,\Z). }
\]
commute, and it follows that $f^*(t(\alpha))=0$.
\end{proof}

\begin{thm}\label{thm:main}
The homomorphism
\[
\Phi_L:\KMW_n(L)\to \HMW^{n,n}(L,\Z)
\]
is an isomorphism for any $n\in\Z$ and any finitely generated field extension $L/k$. 
\end{thm}

\begin{proof}
As in degree $1$, it suffices to prove that $\Phi_L$ is surjective. Let then $\alpha\in \cor k(L,\gm^n)$ be a finite Chow-Witt correspondence supported on $\spec F\subset (\A^1_L)^n$. Such a correspondence is determined by an $n$-uple $(a_1,\ldots,a_n)\in (F^\times)^n$ together with a bilinear form $\phi\in \GW(F,\omega_{F/L})$. Arguing as in Proposition \ref{prop:inductionstep}, we see that such a finite MW-correspondence is of the form $\Tr_{F/L}(\Phi_F(\beta))$ for some $\beta\in \KMW_n(F)$. The result now follows from Theorem \ref{thm:respect}.
\end{proof}


\bibliography{FGW}{}
\bibliographystyle{plain}


\end{document}